\documentclass[reqno, a4paper, 12pt]{amsart}
\usepackage{amsmath}
  \usepackage{paralist}
  \usepackage{graphics} 
  \usepackage{epsfig} 
  \usepackage{tikz-cd}
  \usepackage{quiver}
  \usepackage{xfrac}
  \usepackage{enumitem}
  \usepackage[justification=centering]{caption}
\usepackage{graphicx}  \usepackage{epstopdf}
 \usepackage[colorlinks=true]{hyperref}
\hypersetup{urlcolor=blue, citecolor=red}

\newtheorem{theorem}{Theorem}[section]
\newtheorem{corollary}[theorem]{Corollary}

\newtheorem{lemma}[theorem]{Lemma}
\newtheorem{proposition}[theorem]{Proposition}

\theoremstyle{definition}
\newtheorem{definition}[theorem]{Definition}
\newtheorem{example}[theorem]{Example}
\newtheorem{remark}[theorem]{Remark}

\begin{document}

	\title{Homology and Cohomology of Topological Quandles}
\author{Georgy C Luke}
	\address{Indian Institute of Science Education and research Tirupati,
	Tirupati -517507. Andhra Pradesh, INDIA.}
\email{georgy.c@students.iisertirupati.ac.in}

\author{B. Subhash}
\address{Indian Institute of Science Education and research Tirupati,
	Tirupati -517507. Andhra Pradesh, INDIA.}
\email{subhash@iisertirupati.ac.in}



\maketitle






\begin{abstract}
A homology and cohomology theory for topological quandles are introduced. The relation between these (co)homology groups and quandle (co)homology groups are studied. The 1 - topological quandle cocycles are used to compute state sum invariants corresponding to knot diagrams.
\end{abstract}

\section{Introduction}
Quandles are algebraic structures that corresponds to Reidemeister moves in Knot theory. Quandles that are involutive, are called Keis and structures satisfying all quandle axioms except for idempotency are called racks. The idea of keis was introduced by Takasaki in 1942 and idea of racks was first mentioned in the correspondence between John Conway and Gavin Wraith in 1959. Around 1982, Joyce \cite{joyce1982classifying} and Matveev \cite{matveev1982distributive} introduced the theory of quandles independently. They introduced the notion of fundamental quandle associated to oriented knots and showed that the fundamental quandle is a complete invariant of an oriented knot upto reversed mirror image. 
 After that quandle homology \cite{carter2003quandle}and cohomology theories were introduced modifying the idea of rack cohomology in \cite{fenn1995trunks} and \cite{fenn2004james}.
 \par The notion of topological quandles was introduced by Rubinzstein in 2007\cite{rubinsztein2007topological}. He used topological quandles to associate topological spaces to links, which turned out to be the colouring space aasociated to the link. 
 Later in 2018 continuos cohomology of topological quandles was introduced in \cite{elhamdadi2019continuous}
 \par In this paper we develop homology and cohomology theories for topological quandles and examine it's relation between quandle (co)homology theories. The 1-cocycles are used to compute the state sum invariants of knot diagrams. 
 \par The second section of the article deals with the basics of topological quandles. 
 The chain complexes and cochain complexes of topological quandles are introduced in section three. The properties of topological quandle homology and its relation with quandle homology groups are studied in this section. The fourth section deals with three kinds of topological quandle (co)homology groups analogous to quandle (co)homology groups. The zeroth and first cohomology groups are computed for certain topological quandles in this section. The state sum invariants of knot diagrams are computed using topological quandle 1- cocycles in section five.  
\section{Basic Definitions and Examples}
 \begin{definition}
\par A $quandle$ X is a set with a binary operation $\triangleright$ on X which satisfies the following conditions.
\begin{enumerate}
 \item $x \triangleright x = x$  $\forall$ $x \in X$
 \item The map $\beta_{y} :X\rightarrow X$ defined by $\beta_{y}(x) = x \triangleright y $  for every $y \in X$ is invertible.
 \item $(x \triangleright y)  \triangleright z = (x \triangleright z ) \triangleright (y \triangleright z )  $
\end{enumerate}
 \end{definition}
 
 See \cite{elhamdadi2015quandles} for more details on quandles.
  \begin{definition}
A $topological$ $quandle$ X is a topological space with a quandle structure such that the mapping 
     \begin{align*}
  \triangleright :X \times X \rightarrow X     
\end{align*}       
 is continuous and $\beta_{y}$ is a homeomorphism for every y $\in$ X \cite{elhamdadi2016foundations}.
  \end{definition}
 \begin{definition}
 Let $X$ and $Y$ be $\left( \text{topological} \right)$ quandles. A $\left(\text{continuos}\right)$ map $f:X\rightarrow Y$ is said to be a $\left( \text{topological}\right)$ quandle homomorphism if $f(x_{1}\triangleright x_{2})=f(x_{1})\triangleright f(x_{2})$.
   \end{definition}
Further $f$ is a quandle isomorphism if $f$ is a bijective quandle homomorphism. 
    \begin{definition}
   A map $f$ is a $\text{topological quandle isomorphism}$ if $f$ is a bijective topological quandle homomorphism such that $f^{-1}$ is continous.
 \end{definition}
 \begin{example}
The space $S^{n}\subset \mathbb{R}^{n+1}$ is a topological quandle under the operation $x\triangleright y=2 \langle x,y\rangle y-x$
 \end{example}
   \begin{example}
   There are many topological quandles whose underlying space is a topological group G. For example :\begin{itemize}
       \item The $conjugation$ $quandle$, denoted by $Conj(G)$ is the topological quandle with quandle operation $x\triangleright y=yxy^{-1}$.
       \item The core quandle denoted by $Core(G)$ has the quandle operation $x\triangleright y=yx^{-1}y$.
       \item If there is a homeomorphism $\sigma: G\rightarrow G$, then there exists a topological quandle structure with $x\triangleright y = \sigma(xy^{-1})y$
   \end{itemize}
   \end{example}     
   \begin{example}\label{exagrass}
   \textbf{Grassmanian Manifolds} $Gr_{k}(V)$. Let $V$ be an n dimensional real vector space and $Gr_{k}(V)$ denote the space of all $k$ dimensional subspaces of $V$, where $1 \leq k \leq n$. Let $\beta_{U}:V \rightarrow V$ be a unitary isometry defined as $\beta_{U}(v)=2\displaystyle{\sum_{i=1}^{k}}\langle v,a_{i}\rangle a_{i} - v$ where $U \in Gr_{k}(V)$ and $\lbrace a_{1},a_{2},\cdots,a_{n}\rbrace$ is an orthonormal basis for $U$. Defining $W*U=\beta_{U}(W)$ makes $Gr_{k}(V)$ a topological quandle.
   \end{example}   
   \begin{example}\label{exagrassinfty}
   \textbf{Infinite Grassmanian manifold} $Gr_{k}(R^{\infty})$ Quandle structure on $Gr_{k}(R^{\infty})$ could be defined as in example \ref{exagrass}. For $ k \in \mathbb{N} \setminus \lbrace 0 \rbrace$ and $U$ $\in$ $Gr_{k}(R^{\infty}),$ let $\beta_{U}:R^{\infty} \rightarrow R^{\infty}$ be the unitary isometry defined as $\beta_{U}(v)=2\displaystyle{\sum_{i=1}^{k}}\langle v,a_{i}\rangle a_{i} - v,$ where $\lbrace a_{1},a_{2},\cdots,a_{n}\rbrace$ is an orthonormal basis for $U$. Let $R^{\infty} = U \oplus U^{\perp}$, be an orthogonal decomposition of $R^{\infty}$, any arbitrary element $ R^{\infty} \ni v= v_{1}\oplus v_{2}$ where $v_{1}=\displaystyle{\sum_{i=1}^{k}}\langle v,a_{i}\rangle a_{i}$ and $v_{2}=v-\displaystyle{\sum_{i=1}^{k}}\langle v,a_{i}\rangle a_{i}$. Then $\beta_{U}(v_{1}\oplus v_{2})=v_{1}\oplus -v_{2}$. All the quandle axioms are satisfied, hence forms a topological quandle.
   \end{example}
   
\begin{example}\textbf{Quandles on $\mathbb{R}$}
The binary operation $$x\triangleright y := (tx^{n}+(1-t)y^{n})^{1/n}$$ gives a topological quandle structure on $\mathbb{R}$ when $n$ is odd and $t \in \mathbb{R}\setminus \lbrace 0 \rbrace.$
Note that $x \triangleright x := (tx^{n}+(1-t)x^{n})^{1/n} = x$. 
 Let $\gamma_{y}(x):=\dfrac{(x^{n}-(1-t)y^{n})^{1/n}}{t^{1/n}}$. 
Then one could easily verify $\gamma_{y}\circ \beta_{y}(x) =$ $\beta_{y} \circ \gamma_{y}(x)=x$. Therefore $\beta_{y}$ is a homeomorphism with it's inverse being $\gamma_{y}$. To prove $(x \triangleright y)\triangleright z$ $= (x \triangleright z)\triangleright (y \triangleright z)$ is equivalent to showing $\beta_{z} \circ \beta_{y}(x)$ $=\beta_{y \triangleright z} \circ \beta_{z}(x)$ which is a straightforward computation here.
\end{example}  
\section{(Co)Homology of Topological quandles}
Let $(X,\triangleright)$ be a topological quandle. Let $I_{n}$ denote the collection of $n$ simplices of X, then $I_{n} := \left\lbrace  \sigma_{\alpha}^{n} : \Delta^{n} \rightarrow X \right\rbrace$ forms a quandle under the quandle operation \[\sigma_{\alpha_{1}}^{n}\triangleright \sigma_{\alpha_{2}}^{n}(x):=\sigma_{\alpha_{1}}^{n}(x) \triangleright \sigma_{\alpha_{2}}^{n}(x)\] Define $C_{n}(X):=\displaystyle{\bigoplus_{ \sigma_{\alpha}^{n} \in I_{n}}}\mathbb{Z}_{\sigma_{\alpha}^{n}}$. Let $\sigma_{[x_{1},\cdots,x_{n+1}]}$ denote the $n$ simplex whose $i^{th}$ vertex maps to $x_{i} \in X$ for $i=1,2,\cdots,n+1$.
A boundary map $\partial_{n}:C_n(X)\rightarrow C_{n-1}(X)$ is defined on generators as $ \partial_{n}(\sigma_{[x_{1},\cdots,x_{n+1}]})$
\begin{equation}\label{boundary}
 = \displaystyle{\sum_{i=2}^{n+1}}(-1)^{i}(\sigma_{[x_1,\cdots,\hat{x_i},\cdots,x_{n+1}]}-\sigma_{[x_1,\cdots,\hat{x_i},\cdots, x_{n+1}]}\triangleright\sigma_{[x_i,x_i,\cdots,\hat{x_i},\cdots,x_{n+1}]})
 \end{equation} 
 and extended linearly on $C_n(X)$. In equation \ref{boundary} above, 
$\sigma_{[x_1,\cdots,\hat{x_i},\cdots,x_{n+1}]}$ is the $n-1$ simplex obtained by restricting $\sigma_{[x_1,\cdots,x_{n+1}]}$ on the $i^{th}$ face and $\sigma_{[x_i,x_i,\cdots,\hat{x_i},\cdots,x_{n+1}]}$ is the $n-1$ simplex determined by the composition of $\Delta^{n-1} \xrightarrow{i} \Delta^{n} \xrightarrow{\sigma_{[x_{1},\cdots,x_{n+1}]}} X$ where 
 \begin{align*}
 i:\Delta^{n-1}\ni \displaystyle{\sum_{j=0}^{n-1}}t_{j}\bar{e_{j}}\rightarrow (\displaystyle{\sum_{k=0}^{i-2}}t_{k})e_{i-1}+\displaystyle{\sum_{j=i-1}^{n-1}}t_{j}e_{j+1},
 \end{align*} 
 here $\lbrace \bar{e_{j}} \rbrace_{j=0}^{n-1},\:\text{and}\: \lbrace e_{j} \rbrace_{j=0}^{n}$ are vertices of $\Delta^{n-1}$ and $\Delta^{n}$ respectively. We will denote the $m$ simplex
 $\sigma_{[x_{1},\cdots,x_{m+1}]}\triangleright \sigma_{[y_{1},\cdots,y_{m+1}]}$ by $\sigma_{[x_{1}\triangleright y_{1},\cdots,x_{m+1} \triangleright y_{m+1}]}$ for the rest of the article.
\begin{proposition}
The composition  $C_{n}(X) \xrightarrow{\partial_{n}} C_{n-1}(X) \xrightarrow{\partial_{n-1}} C_{n-2}(X)$ is zero.
\end{proposition}
\begin{proof}
\begin{equation*}
\begin{split}
&\partial_{n-1}(\partial_{n}(\sigma_{[x_{1},\cdots,x_{n+1}]})) \\
&= \partial_{n-1}(\displaystyle{\sum_{i=2}^{n+1}}(-1)^{i}(\sigma_{[x_{1},\cdots,\hat{x_{i}},\cdots,x_{n+1}]}
-\sigma_{[x_{1}\triangleright x_{i},x_{2}\triangleright x_{i},\cdots,\hat{x_{i}},\cdots,x_{n+1}]})\\
&= \displaystyle{\sum_{j<i,i,j \in \lbrace2,\cdots,n\rbrace}}(-1)^{i}(-1)^{j}(\sigma_{[x_{1},\cdots,\hat{x_{j}},\cdots,\hat{x_{i}},\cdots,x_{n+1}]}\\
&-\sigma_{[x_{1},\cdots,\hat{x_{j}},\cdots,\hat{x_{i}},\cdots,x_{n+1}]}\triangleright\sigma_{[x_{j},x_{j},\cdots,\hat{x_{j}},x_{j+1},\cdots,\hat{x_{i}},\cdots,x_{n+1}]})\\
&+ \displaystyle{\Sigma_{j>i,i,j \in \lbrace2,\cdots,n\rbrace}}(-1)^{(i)}(-1)^{(j-1)}(\sigma_{[x_{1},\cdots,\hat{x_{i}},\cdots,\hat{x_{j}},\cdots,x_{n+1}]}\\
&-\sigma_{[x_{1},\cdots,\hat{x_{i}},\cdots,\hat{x_{j}},\cdots,x_{n+1}]}\triangleright\sigma_{[x_{j},x_{j},\cdots,\hat{x_{i}},x_{j},\cdots,\hat{x_{j}},\cdots,x_{n+1}]})\\
&+ \displaystyle{\sum_{j<i,i,j \in \lbrace2,..n\rbrace}}(-1)^{(i+1)}(-1)^{(j)}(\sigma_{[x_{1}\triangleright x_{i},\cdots,\hat{x_{j}\triangleright x_{i}},\cdots,x_{i-1}\triangleright x_{i},\hat{x_{i}},\cdots,x_{n+1}]}\\
&-\sigma_{[x_{1}\triangleright x_{i},\cdots,\hat{x_{j}\triangleright x_{i}},\cdots,x_{i-1}\triangleright x_{i},\hat{x_{i}},\cdots,x_{n+1}]}\triangleright\sigma_{[x_{j}\triangleright x_{i},x_{j}\triangleright x_{i},\cdots,\hat{x_{j}\triangleright x_{i}},\cdots,x_{i-1}\triangleright x_{i},\hat{x_{i}},\cdots,x_{n+1}]})\\
&+ \displaystyle{\sum_{j>i,i,j \in \lbrace2,\cdots,n\rbrace}}(-1)^{(i+1)}(-1)^{(j-1)}(\sigma_{[x_{1}\triangleright x_{i},\cdots,\hat{x_{i}},\cdots,x_{j-1},\hat{x_{j}},\cdots,x_{n+1}]}\\
&-\sigma_{[x_{1}\triangleright x_{i},\cdots,\hat{x_{i}},\cdots,x_{j-1}\hat{x_{j}},\cdots,x_{n+1}]}\triangleright\sigma_{[x_{j},x_{j},\cdots,\hat{x_{i}},x_{j},\cdots,\hat{x_{j}},\cdots,x_{n+1}]})\\
\end{split}
\end{equation*}
All these terms cancel each other as $I_{n}$ is a quandle. So 
$\partial^{2}=0$
\end{proof}
\noindent Define the $n^{th}$ chain group with coefficient group $A$ as $C_{n}(X,A):= C_{n}(X)\otimes A$ and boundary map $\partial_{A}^{n}:=\partial_{n} \otimes id$. The $n^{th}$ cochain group $C^{n}(X,A):= Hom(C_{n},A)$ with coboundary map $\delta^{n}:C^{n}(X,A)\rightarrow C^{n+1}(X,A)$ where  $\delta^{n}\varphi(\sigma):=\varphi (\partial_{n+1}\sigma)$ where $\varphi \in C^{n}(X,A)\; \text{and}\; \sigma \in C_{n+1}(X). $
\begin{definition}
The $n^{th}$ homology and cohomology group of above (co)chain complex with coefficients in $A$ are denoted as $H_{n}(X,A)$ and $H^{n}(X,A)$ respectively. The cycle and boundary (cocycle and coboundary) groups are denoted by $Z_{n}(X,A)$ and $B_{n}(X,A)$ ($Z^{n}(X,A)$ and $B^{n}(X,A)$)  respectively. Hence $H_{n}(X,A)=\dfrac{Z_{n}(X,A)}{B_{n}(X,A)}$ and $H^{n}(X,A)=\dfrac{Z^{n}(X,A)}{B^{n}(X,A)}.$
\end{definition}
\begin{lemma}
Let $f : (X,\triangleright) \rightarrow (Y,\triangleright)$ be topological quandle homomorphism, then $f$ induces a group homomorphism between $H_{n}(X)$ and $H_{n}(Y).$
\end{lemma}
\begin{proof}
As $f : (X,\triangleright) \rightarrow (Y,\triangleright)$ is a continuos map, we can define a chain map $f_{\sharp}:C_n(X) \rightarrow C_n(Y)$ as $f_{\sharp}(\sigma_{[x_{1},\cdots,x_{n+1}]}):=f\circ \sigma_{[x_{1},\cdots,x_{n+1}]}$ on the generators and extend it linearly on $C_{n}(X)$ as $f_{\sharp}(\sum k_{\alpha} \sigma_{\alpha}^{n})=\sum k_{\alpha}(f\circ \sigma_{\alpha})$.
Note that
\begin{equation*}
\begin{split}
\partial(f_{\sharp}(\sigma_{[x_{1},\cdots,x_{n+1}]}))=& \partial(f\circ \sigma_{[x_{1},\cdots,x_{n+1}]})\\
=& \displaystyle{\sum_{i=2}^{n+1}}(-1)^{(i)}(f \circ \sigma_{[x_{1},\cdots,\hat{x_{i}},\cdots,x_{n+1}]})\\
&- f\circ \sigma_{[x_{1}\triangleleft x_{i},x_{2}\triangleright x_{i},\cdots,\hat{x_{i}},\cdots,x_{n+1}]}\\
=& f_{\sharp}\partial(\sigma_{[x_{1},\cdots,x_{n+1}]})
\end{split}
\end{equation*} 
and therefore the following diagram commutes.
\begin{equation*}
\begin{tikzcd}
\cdots\arrow[r," "]& C_{n+1}(X) \arrow[r, "\partial"] \arrow[d, "f_{\sharp}"]
& C_{n}(X) \arrow[r,"\partial"] \arrow[d, "f_{\sharp}" ] & C_{n-1}(X)\arrow[r," "] \arrow[d, "f_{\sharp}" ] & \cdots\\
\cdots \arrow[r," "]& C_{n+1}(Y)\arrow[r,"\partial"] & C_{n}(Y)\arrow[r,"\partial"]& C_{n-1}(Y)\arrow[r," "]&\cdots 
\end{tikzcd}
\end{equation*}
So there is an induced group homomorphism $f_{*}:H_{n}(X)\rightarrow H_{n}(Y)$ since $f_{\sharp}(Z_{n}(X))\subset Z_{n}(Y)$ and $f_{\sharp}(B_{n}(X))\subset B_{n}(Y).$
\end{proof}
\begin{corollary}
If $(X,\triangleright)$ and $(Y,\triangleright)$ are isomorphic as topological quandles, then $H_{n}(X)$ and $H_{n}(Y)$ are isomorphic as groups.
\end{corollary}
\begin{proof}
Suppose $f: (X,\triangleright) \rightarrow (Y,\triangleright)$ is a topological quandle isomorphism by previous lemma  $f_{*}:H_{n}(X)\rightarrow H_{n}(Y)$ is a group homomorphism. Similarly there exists a group homomorphism  $f^{-1}_{*}:H_{n}(Y)\rightarrow H_{n}(X)$. Note that  $f_{*}\circ f^{-1}_{*}=id_{H_{n}(Y)}$ and $f^{-1}_{*}\circ f_{*}=id_{H_{n}(X)}$ since $(f\circ f^{-1})_{*}=f_{*}\circ f^{-1}_{*}$ and $(f^{-1}\circ f)_{*}=f^{-1}_{*}\circ f_{*}.$
\end{proof}
\begin{theorem}
Let $X$ be a topological quandle.
\begin{enumerate}
\item If $X$ has discrete topology,then $H_{n}(X)=C_{n}(X).$
\item If $X$ has a trivial quandle structure,then $H_{n}(X)=C_{n}(X).$
\end{enumerate}
\end{theorem}
\begin{proof}
 $(1)$ If $X$ has discrete topology, the only connected sets of $X$ are one points sets. As the image of connected sets are connected under continuos maps, all $n$ simplices should be constant. Suppose $\sigma_{[x,\cdots,x]}:\Delta^{n} \rightarrow X$ be the constant $n$ simplex mapping to $x\in X$. Then $\partial(\sigma_{[x,\cdots,x]})=0$. Therefore $B_{n}(X)=0 $ and $Z_{n}(X)=C_{n}(X). $ So $H_{n}(X)=C_{n}(X).$
 \vspace*{3mm}\\
 $(2)$ If $X$ has a trivial quandle structure then $\partial_{n}(\sigma_{[x_{1},\cdots,x_{n+1}]})$ equals to  $\displaystyle{\sum_{i=2}^{n+1}}(-1)^{i}(\sigma_{[x_{1},\cdots,\hat{x_{i}},\cdots,x_{n+1}]}-\sigma_{[x_{1},\cdots,\hat{x_{i}},\cdots,x_{n+1}]}\triangleright\sigma_{[x_{i},x_{i},\cdots,\hat{x_{i}},\cdots,x_{n+1}]})$ which equals to zero since $\sigma_{[x_{1},\cdots,\hat{x_{i}},\cdots,x_{n+1}]}\triangleright\sigma_{[x_{i},x_{i},\cdots,\hat{x_{i}},\cdots,x_{n+1}]}=\sigma_{[x_{1},\cdots,\hat{x_{i}},\cdots,x_{n+1}]}$. Therefore $B_{n}(X)=0 $ and $Z_{n}(X)=C_{n}(X) $. So $H_{n}(X)=C_{n}(X).$
\end{proof}
 \begin{example}
 Consider the dihedral quandle $\mathbb{Z}_{3}$ with indiscrete topology. Let $\sigma_{\left[ a,b\right] } \in C_{1}(X)$
 be the path in $\mathbb{Z}_{3}$ connecting $a,b \in \mathbb{Z}_{3}$. Then $\partial_{1}(\sigma_{\left[ a,b\right] })=\sigma_{a}-\sigma_{a \triangleright b}$ where  $\sigma_{a},\sigma_{a \triangleright b}$ are the constant $0$ simplices. Note that $\partial_{1}(\sigma_{\left[ a,a\right] })=0$ for $a \in \mathbb{Z}_{3}$. Hence $B_{0}(\mathbb{Z}_{3})=\left\langle \sigma_{1}-\sigma_{2},\sigma_{2}-\sigma_{3},\sigma_{3}-\sigma_{1}\right\rangle $, also $H_{0}(\mathbb{Z}_{3})=\dfrac{\mathbb{Z}_{\sigma_{1}}\oplus\mathbb{Z}_{\sigma_{2}}\oplus\mathbb{Z}_{\sigma_{3}}}{\left\langle \sigma_{1}-\sigma_{2},\sigma_{2}-\sigma_{3},\sigma_{3}-\sigma_{1}\right\rangle} = \mathbb{Z}.$
 \end{example} 
 \begin{example}
 Consider $\mathbb{Z}_{3}$ with indiscrete topology and with the quandle operation defined as in the following table
 \begin{center}
 $
 M=\begin{bmatrix}
 1&1&1\\
 3&2&2\\
 2&3&3\\
 \end{bmatrix}
 $
 \end{center}
 Then $\partial_{1}(\sigma_{\left[ a,a\right] })=0$, $a \neq b$ if and only if $(a,b)=(2,1)$ or $(a,b)=(3,1)$. Then 
 $B_{0}(\mathbb{Z}_{3})=\left\langle \sigma_{3}-\sigma_{2}\right\rangle$. Therefore $H_{0}(\mathbb{Z}_{3})=\dfrac{\mathbb{Z}_{\sigma_{1}}\oplus\mathbb{Z}_{\sigma_{2}}\oplus\mathbb{Z}_{\sigma_{3}}}{\left\langle \sigma_{2}-\sigma_{3}\right\rangle} = \mathbb{Z}_{\sigma_{1}} \oplus\mathbb{Z}_{\sigma_{2}}$
 \end{example}
 \noindent Now we observe some results similiar to that of results of usual integral homology of topological spaces. 
    \begin{definition}
   A quandle $X$ is said to be indecomposable, if for every pair $x,y\in X$ there exists $y_{1},\cdots,y_{n}\in X$
   and $e_{1},\cdots,e_{n}\in \lbrace-1,1\rbrace$ such that $x=y \triangleright^{e_{1}}y_{1}\triangleright^{e_{2}}\cdots\triangleright^{e_{n}}y_{n}.$
   \end{definition}
 \begin{theorem}
 If $X$ is a path connected, indecomposable topological quandle, then $H_{0}(X)=\mathbb{Z}.$ 
 \end{theorem}
 \begin{proof}
 As $X$ is path connected, for every $(a,b)\in X \times X$ there exists atleast one path in $X$ connecting $a$ and $b$. Then $B_{0}(X)=\lbrace\partial(\sigma_{[a,b]})=\sigma_{a}-\sigma_{a\triangleright b} \, |  \, \forall $ path  $\sigma_{[a,b]}$ between $a$ and $b$, $\forall \, (a,b)\in X \times X\rbrace$. Therefore in $H_{0}(X), \,\left[ \sigma_{a}\right]=\left[ \sigma_{a\triangleright b}\right] $. As $X$ is indecomposable $ \left[ \sigma_{x}\right] =\left[ \sigma_{y}\right]$ for all $x$ and y. Hence $H_{0}(X)=\mathbb{Z}.$ 
 \end{proof}
 \begin{lemma}
Let $\lbrace X_{\alpha} \rbrace$ be the decomposition of a topological quandle $X$ into its path components. Then $H_{n}(X)=\displaystyle{\bigoplus_{\alpha}}H_{n}(X_{\alpha})$ 
 \end{lemma}
 \begin{proof}
 As image of each $n$ simplex $\sigma_{[x_{1},\cdots,x_{n+1}]}$ lies in some $X_{\alpha}$, $C_{n}(X)=\displaystyle{\bigoplus_{\alpha}} C_{n}(X_{\alpha})$. Also $\partial_{n}^{\alpha}(C_{n}(X_{\alpha}))\subset C_{n-1}(X_{\alpha})$, where $\partial_{n}^{\alpha}:=\partial_{n}\vert_{C_{n}(X_{\alpha})}$. Here $Ker\partial_{n}=\displaystyle{\bigoplus_{\alpha}}Ker\partial_{n}^{\alpha}$ and $Im\partial_{n+1}=\displaystyle{\bigoplus_{\alpha}}Im\partial_{n+1}^{\alpha}$. Hence $H_{n}(X)=\dfrac{Ker\partial_{n}}{Im\partial_{n+1}}=\dfrac{\displaystyle{\bigoplus_{\alpha}}Ker\partial_{n}^{\alpha}}{\displaystyle{\bigoplus_{\alpha}}Im\partial_{n+1}^{\alpha}} =\displaystyle{\bigoplus_{\alpha}}H_{n}(X_{\alpha})$
 \end{proof}
 \begin{corollary}
 Let $\lbrace X_{\alpha} \rbrace$ be the decomposition of a topological quandle $X$ into its path components such that each  $ X_{\alpha} $ is indecomposable. Then $H_{0}(X)=\displaystyle{\bigoplus_{\alpha}}\mathbb{Z}_{\alpha}$ 
 \end{corollary}
 \subsection{Relation between Quandle Homology}
 Quandle homology is studied and invariants of knots and knotted surfaces are constructed using quandle cocycles in \cite{carter2003quandle}. The $n^{th}$ $rack$ $homology$ group $H_{n}^{R}(X)$ is the $n^{th}$ homology group of the chain complex $C_{*}^{R}(X)$ which is the free abelian group generated by the $n$ tuples $(x_{1},\cdots,x_{n})\in X^{n}$ and whose boundary maps are defined as 
 \begin{equation*}
 \begin{split}
 \partial_{n}^{R}(x_{1},\cdots,x_{n}) &= \displaystyle{\sum_{i=2}^{n}}(-1)^{(i)}(x_{1},\cdots,\hat{x_{i}},\cdots,x_{n})\\
 &-(x_{1}\triangleright x_{i},\cdots,x_{i-1}\triangleright x_{i},\hat{x_{i}},\cdots,x_{n})
 \end{split}
 \end{equation*}
 The $n^{th}$ $degenerate$ $homology$ group $H_{n}^{D}(X)$ is the $n^{th}$ homology group of the chain complex $C_{*}^{D}(X)$ where $C_{n}^{D}(X)$ is the subgroup of $C_{n}^{R}(X)$ generated by $n$ tuples $(x_{1},\cdots,x_{n})$ with $x_{i}=x_{i+1}$ for $i \in \lbrace1,2,\cdots,n-1\rbrace $.\\
 Similarly $n^{th}$ $quandle$ $homology$ group $H_{n}^{Q}(X)$ is the $n^{th}$ homology group of the chain complex $C_{*}^{Q}(X)$ where $C_{n}^{Q}(X)= C_{n}^{R}(X)/C_{n}^{D}(X)$. 
 \begin{proposition}
 There exists a group homomorphism from $H_{n}(X)$ to $H_{n+1}^{R}(X).$
 \end{proposition}
 \begin{proof}
 Define the map $\psi_{n}:C_{n}(X)\rightarrow C_{n+1}^{R}(X)$ on generators as:
 \begin{equation*}
 \psi_{n}(\sigma_{[x_{1},\cdots,x_{n+1}]}) =(x_{1},\cdots,x_{n+1})
 \end{equation*}
 and extend it linearly over $C_{n}(X)$. Then the following diagram commutes 
 \begin{equation*}
\begin{tikzcd}
\cdots C_{n}(X) \arrow[r,"\partial_{n}"] \arrow[d, "\psi_{n}" ] & C_{n-1}(X)\cdots \arrow[d, "\psi_{n-1}" ] \\
\cdots C_{n+1}^{R}(X)\arrow[r,"\partial_{n+1}^{R}"]& C_{n}^{R}(X)\cdots 
\end{tikzcd}
\end{equation*}
So there is a well defined group homomorphism $\psi_{n*} $from $H_{n}(X)$ to $H_{n+1}^{R}(X)$ mapping 
$[\sigma_{[x_{1},\cdots,x_{n+1}]}]$ to $[(x_{1},\cdots,x_{n+1})]$.
 \end{proof}
 \begin{corollary}
 Let $X$ be a path connected topological quandle. Then $H_{0}(X)$ and $H_{1}^{R}(X)$ are isomorphic. 
 \end{corollary}
 \begin{proof}
 Note that $H_{0}(X)=\dfrac{C_{0}(X)}{B_{0}(X)}$ and $H_{1}^{R}(X)=\dfrac{C_{1}^{R}(X)}{B_{1}^{R}(X)}$. Then $\psi_{0*}:H_{0}(X)\rightarrow H_{1}^{R}(X)$ defined as $\psi_{0*}(\overline{\displaystyle{\sum_{i}}c_{i}\sigma_{[x_{i}]}}) :=\overline{\displaystyle{\sum_{i}}c_{i}(x_{i})}$ is a surjective group homomorphism. Now define $\chi_{j}:C_{j+1}^{R}(X)\rightarrow C_{j}(X)$ for $j=1,2$ as $\chi_{0}(x_{1}):=\sigma_{[x_{1}]}$, $\chi_{1}(x_{1},x_{2}):=\sigma_{[x_{1},x_{2}]}$ and extend it linearly. For defining $\chi_{1}$ we fix a path $\sigma_{[x_{1},x_{2}]}$ for each $(x_{1},x_{2})\in X\times X$ and this is possible as $X$ is path connected.
 Now to check the injectivity, assume 
 \begin{equation}
 \begin{split}
 \psi_{0*}(\overline{\displaystyle{\sum_{i}}c_{i}\sigma_{[x_{i}]}}) &=\psi_{0*}(\overline{\displaystyle{\sum_{j}}a_{j}\sigma_{[y_{j}]}})\\
 \Rightarrow \overline{\displaystyle{\sum_{i}}c_{i}x_{i}} &=\overline{\displaystyle{\sum_{j}}a_{j}y_{j}}\\
 \Rightarrow \displaystyle{\sum_{i}}c_{i}x_{i} &-\displaystyle{\sum_{j}}a_{j}y_{j} \in B_{1}^{R}(X)\\
\end{split} 
 \end{equation}
 Then $\chi_{0}(\displaystyle{\sum_{i}}c_{i}x_{i} -\displaystyle{\sum_{j}}a_{j}y_{j} )= \displaystyle{\sum_{i}}c_{i}\sigma_{[x_{i}]} -\displaystyle{\sum_{j}} a_{j} \sigma_{[y_{j}]} \in B_{0}(X)$ as the following diagram commutes.
 \begin{equation*}
 \begin{tikzcd}
\cdots C_{2}^{R}(X) \arrow[r,"\chi_{1}"] \arrow[d, "\partial_{2}^{R}" ] & C_{1}(X) \cdots \arrow[d, "\partial_{1}" ] \\
\cdots C_{1}^{R}(X)\arrow[r,"\chi_{0}"]& C_{0}(X)\cdots  
\end{tikzcd}
\end{equation*}
So it follows that $\overline{\displaystyle{\sum_{i}}c_{i}\sigma_{[x_{i}]}} =\overline{\displaystyle{\sum_{j}}a_{j}\sigma_{[y_{j}]}}$
 \end{proof}
 \section{Three kinds of Co(homology) groups for topological quandles}
 \begin{definition}
 Let $\bar{C}_{n}^{R}(X):=C_{n}(X)/\sim  $ and  $n$ simplices $\sigma \sim \beta$ if and only if $i^{th}$ vertex of $\sigma$ equals to $i^{th}$ vertex of $\beta $ for every $i=1,\cdots,n+1$. The homology group of this chain complex is denoted by $\bar{H}_{n}^{R}(X)$. Put $\bar{C}_{n}^{D}(X)$ be the subgroup of $\bar{C}_{n}^{R}(X)$ generated by  $n$ simplices $\sigma_{[x_{1},\cdots,x_{n+1}]}$ with $x_{i}=x_{i+1}$ for $i\in\lbrace1,2,\cdots,n\rbrace$ for $n\geq 1$, else $\bar{C}_{n}^{D}(X)=0$. Take the boundary operator to be restriction of $\partial_{n}$ on $\bar{C}_{n}^{D}(X)$. One can see that  $\partial_{n}(\bar{C}_{n}^{D}(X))\subset \bar{C}_{n-1}^{D}(X)$. Define $\bar{C}_{n}^{Q}(X):=\bar{C}_{n}^{R}(X)/\bar{C}_{n}^{D}(X) $ with boundary operator being the induced operator. Denote $H_{n}(\bar{C}_{*}^{D}(X))$ by $\bar{H}_{n}^{D}(X)$ and $H_{n}(\bar{C}_{*}^{Q}(X))$ by $\bar{H}_{n}^{Q}(X)$. The corresponding $n^{th}$ cohomology groups with coefficients in $\mathbb{Z}$ is denoted by $\bar{H}^{n}_{W}(X)$ for $W=R,D$ or $Q$.
 \end{definition}
 \begin{theorem}
 If $X$ is a topological quandle with indiscrete topology, then  $\bar{H}_{n}^{W}(X)=H_{n+1}^{W}(X)$ where W=R,D and Q.
 \end{theorem}
 \begin{proof}
  Define $\psi_{n}:\bar{C}^{W}_{n}(X) \rightarrow C^{W}_{n+1}(X)$ on generators as $$\psi_{n}(\sigma_{[x_{1},\cdots,x_{n+1}]}):=(x_{1},\cdots,x_{n+1})$$ and extend it linearly. Similarly define $\psi_{n}^{-1}:{C}^{W}_{n+1}(X) \rightarrow \bar{C}^{W}_{n}(X)$ as $\psi_{n}(x_{1},\cdots,x_{n+1})=\sigma_{[x_{1},\cdots,x_{n+1}]}$
 There is a one to one correspondence between $\bar{C}_{n}^{W}(X)$ and $C_{n+1}^{W}(X)$ and the following diagram commutes.
 \[\begin{tikzcd}
	{\bar{C}_{n}^{W}(X)} && {C_{n+1}^{W}(X) } && {\bar{C}_{n}^{W}(X)} \\
	\\
	{\bar{C}_{n-1}^{W}(X)} && {C_{n}^{W}(X)} && {\bar{C}_{n-1}^{W}(X)}
	\arrow["{\psi_{n}}", from=1-1, to=1-3]
	\arrow["{\psi_{n}^{-1}}", from=1-3, to=1-5]
	\arrow["{\partial_{n}}"', from=1-1, to=3-1]
	\arrow["{\partial_{n+1}^{W}}"', from=1-3, to=3-3]
	\arrow["{\partial_{n}}", from=1-5, to=3-5]
	\arrow["{\psi_{n-1}}", from=3-1, to=3-3]
	\arrow["{\psi_{n-1}^{-1}}", from=3-3, to=3-5]
\end{tikzcd}\]
This is because
\begin{equation*}
\begin{split}
 & \partial_{n+1}^{W} \circ \psi_{n}(\sigma_{[x_{1},\cdots,x_{n+1}]})\\
 & =\partial_{n+1}^{W}(x_{1},\cdots,x_{n+1})  \\
 &=\displaystyle{\sum_{i=2}^{n}}(-1)^i((x_{1},\cdots,\hat{x_{i}},\cdots,x_{n}) -(x_{1}\triangleright x_{i},\cdots,x_{i-1}\triangleright x_{i},\hat{x_{i}},\cdots,x_{n}))\\
 &=\psi_{n-1}(\displaystyle{\sum_{i=2}^{n+1}}(-1)^i(\sigma_{[x_{1},\cdots,\hat{x_{i}},\cdots,x_{n+1}]}
-\sigma_{[x_{1}\triangleright x_{i},x_{2}\triangleright x_{i},\cdots,\hat{x_{i}},\cdots,x_{n+1}]}))\\
&=\psi_{n-1}\circ \partial_{n}(\sigma_{[x_{1},\cdots,x_{n+1}]})\\
\end{split}    
\end{equation*}
Similarly $\partial_{n} \circ \psi_{n}^{-1}=\psi_{n-1}^{-1}\circ \partial_{n+1}^{W}$ and the theorem follows. 
 \end{proof}
 \begin{theorem}
 Let $X$ be a topological quandle
 \label{thm:cohomology}
\begin{enumerate}
\item If X has discrete topology then $\bar{H}^{n}_{Q}(X)=0,\bar{H}^{n}_{W}(X)=\bar{C}^{n}_{W}(X)$ for $n\geq 1, W=R,D$ and $\bar{H}^{0}_{Q}(X)=\bar{H}^{0}_{R}(X)=\bar{C}^{0}_{R}(X)$
\item If X has indiscrete topology then $\bar{H}^{n}_{W}(X)=H^{n+1}_{W}(X)$ for $n\geq 0$ and $W=R,D$ and $Q.$
\item If X has an underlying trivial quandle structure then  $\bar{H}^{n}_{W}(X)=\bar{C}^{n}_{W}(X)$ for $W=R,D$ and $Q.$
\end{enumerate} 
 \end{theorem}
 \begin{proof}
 $(1)$ Consider the boundary homomorphism $\delta^{n}$ from $\bar{C}^{n}_{W}(X)$ to $\bar{C}^{n+1}_{W}(X)$ for $W=R$ and $D$. Suppose $f \in \bar{C}^{n}_{W}(X) $. As $\bar{C}_{n}^{W}(X)$ is the free group generated by constant $n$ simplices, $\delta^{n}f=0$. This is because $\delta^{n}f(\sigma_{[x_{i},\cdots,x_{i}]})=f(\partial_{n+1}\sigma_{[x_{i},\cdots,x_{i}]})=0$. Therefore $\bar{H}^{n}_{W}(X)= \bar{C}^{n}_{W}(X)$ for $n \geq 0$. But $\bar{C}^{n}_{Q}(X)=0$ for $n \geq 1$ as $\bar{C}^{n}_{R}(X)=\bar{C}^{n}_{D}(X)$ for $n \geq 1$. Since $\bar{C}^{0}_{D}(X)=0$, $\bar{C}^{0}_{Q}(X)=\bar{C}^{0}_{R}(X)$ and hence the statement follows. 
 \vspace*{3mm}\\
 $(2)$ Let $f \in \bar{C}^{n}_{W}(X)$. Define $\psi_{n}:\bar{C}^{n}_{W}(X) \rightarrow C^{n+1}_{W}(X)$ as $\psi_{n}(f)=f_{\sharp}$ where $f_{\sharp} \in C^{n+1}_{W}(X)$ is defined on generators as $$f_{\sharp}(x_{1},x_{2},\cdots,x_{n+1}):=f(\sigma_{[x_{1},x_{2},\cdots,x_{n+1}]})$$ and extended linearly on $C_{n+1}^{W}(X)$. There is a one to one correspondence between $\bar{C}^{n}_{W}(X)$ and $C^{n+1}_{W}(X)$.
\begin{equation*}
 \begin{tikzcd}
f\in \bar{C}^{n}_{W}(X) \arrow[r,"\psi_{n}"] \arrow[shift left=2,d,"\delta^{n}"] &  C^{n+1}_{W}(X)\ni f_{\sharp} \arrow[shift right=5,d,"\delta^{n+1}_{W}"] \\
\delta^{n}f \in \bar{C}^{n+1}_{W}(X) \arrow[r,"\psi_{n+1}"]& C^{n+2}_{W}(X)\ni \delta^{n+1}_{W}f_{\sharp} 
\end{tikzcd}
\end{equation*}
The above diagram commutes as $\delta^{n+1}_{W} \circ \psi_{n}=\psi_{n+1}\circ \delta^{n}$. This is because 
\begin{equation*}
\begin{split}
&(\delta^{n}f)_{\sharp}(x_{1},\cdots,x_{n+2})\\
&=\delta^{n}f(\sigma_{[x_{1},\cdots,x_{n+2}]})\\
&=f(\partial_{n+1}\sigma_{[x_{1},\cdots,x_{n+2}]})\\
&=f(\sum_{i=2}^{n+2}(-1)^{(i)}(\sigma_{[x_{1},\cdots,\hat{x_{i}},\cdots,x_{n+2}]}-\sigma_{[x_{1},\cdots,\hat{x_{i}},\cdots,x_{n+2}]}\triangleright\sigma_{[x_{i},x_{i},\cdots,\hat{x_{i}},\cdots,x_{n+2}]}))\\
&=\sum_{i=2}^{n+2}(-1)^{(i)}f(\sigma_{[x_{1},\cdots,\hat{x_{i}},\cdots,x_{n+2}]}-f(\sigma_{[x_{1},\cdots,\hat{x_{i}},\cdots,x_{n+2}]}\triangleright \sigma_{[x_{i},x_{i},\cdots,\hat{x_{i}},\cdots,x_{n+2}]}))\\
&=\sum_{i=2}^{n+2}(-1)^{(i)}(f_{\sharp}(x_{1},\cdots,\hat{x_{i}},\cdots,x_{n+2})\\
&-f_{\sharp}((x_{1},..\hat{x_{i}},\cdots,x_{n+2})\triangleright (x_{i},x_{i},\cdots,\hat{x_{i}},\cdots,x_{n+2})))\\
&=\delta^{n+1}_{W}f_{\sharp}(x_{1},\cdots,x_{n+2})\\
\end{split}
\end{equation*} 
$(3)$ Note that the coboundary operator $\delta^{n}=0$ as the boundary operator $\partial_{n+1}:\bar{C}_{n+1}^{W}(X)\rightarrow \bar{C}_{n}^{W}(X)$ is zero for a trivial quandle structure.
 \end{proof}
 \begin{remark}
 Analogous to 1 and 3 above the exact statements holds for homology groups.
 \end{remark}
\begin{theorem}
If $X$ is a topological quandle then there exists a long exact sequence of homology groups.
\begin{equation*}
\cdots \xrightarrow{\partial \ast} \bar{H}_{n}^{D}(X) \xrightarrow{\text{i} \ast} \bar{H}_{n}^{R}(X) \xrightarrow{\text{j} \ast} \bar{H}_{n}^{Q}(X) \xrightarrow{\partial \ast} \bar{H}_{n-1}^{D}(X)\rightarrow\cdots
\end{equation*}
\end{theorem}
\begin{proof}
The following is a short exact sequence of chain complexes.
\begin{equation*}
0 \rightarrow \bar{C}_{\ast}^{D}(X) \xrightarrow{\text{i} } \bar{C}_{\ast}^{R}(X) \xrightarrow{\text{j} } \bar{C}_{\ast}^{Q}(X) \rightarrow 0
\end{equation*}
This gives rise to the long exact sequence of homology groups in theorem.
\end{proof}
\begin{theorem}
\textbf{(Universal Coefficient Theorem)}We have the following split short exact sequences.
\begin{equation*}
0\rightarrow \bar{H}_{n}^{W}(X)\otimes G \rightarrow \bar{H}_{n}^{W}(X,G) \rightarrow Tor(\bar{H}_{n-1}^{W}(X),G) \rightarrow 0
\end{equation*}
\begin{equation*}
0\rightarrow Ext(\bar{H}_{n-1}^{W}(X), G) \rightarrow \bar{H}^{n}_{W}(X,G) \rightarrow Hom(\bar{H}_{n}^{W}(X),G) \rightarrow 0
\end{equation*} for $W=D,R,Q$
\end{theorem}
\begin{proof}
The chain complexes $\lbrace \bar{C}_{n}^{W}(X)\rbrace$ and $\lbrace \bar{C}^{n}_{W}(X)\rbrace$ are free abelian and the theorem follows.
\end{proof}
\noindent  Now for calculating the cohomology we define $characteristic\, functions$ as in \cite{carter2003quandle}. 
  \begin{equation*}
  \chi_{\sigma_{\left[x_{1},x_{2},\cdots,x_{n+1}\right]}}(\sigma_{\left[y_{1},y_{2},\cdots,y_{n+1}\right]})=
  \begin{cases}
  1 & \text{if}\: \sigma_{\left[x_{1},x_{2},\cdots,x_{n+1}\right]}=\sigma_{\left[y_{1},y_{2},\cdots,y_{n+1}\right]}\\
  0 & \text{otherwise}
  \end{cases}
  \end{equation*}
  where $\sigma_{\left[x_{1},x_{2},\cdots,x_{n+1}\right]}$ is a generator of $\bar{C}_{n}^{W}(X)$. 
  \vspace*{3mm}\\
  Any $f \in  \bar{C}^{n}_{W}(X)$ can be written as a finite linear combination of characteristic functions where $X$ is a finite quandle. 
  Also any $f\in \bar{Z}^{n}_{W}(X)$ satisfies the condition $\delta^{n}f=0$ which simplifies to the criteria 
  \begin{equation*}
  \begin{split}
  f(\partial \sigma_{\left[x_{1},x_{2},\cdots,x_{n+2} \right] })&= \displaystyle{\sum_{i=2}^{n+2}}(-1)^{i}(f(\sigma_{[x_{1},\cdots,\hat{x_{i}},\cdots,x_{n+2}]})\\
  & -f(\sigma_{[x_{1},\cdots,\hat{x_{i}},\cdots,x_{n+2}]}\triangleright \sigma_{[x_{i},x_{i},\cdots,\hat{x_{i}},\cdots,x_{n+2}]})).\\
  \end{split}
  \end{equation*}
  Now we compute the homology and cohomology groups of two topological quandles whose underlying quandle structure is given in the table below.
\begin{center}
 $
 M=\begin{bmatrix}
 1&1&1\\
 3&2&2\\
 2&3&3\\
 \end{bmatrix}
 $
 \end{center}
 Consider the topological quandle $X=\lbrace1,2,3\rbrace$ with quandle operation table $M$ and whose underlying topology is $\tau=\lbrace\emptyset, \lbrace 1\rbrace ,\lbrace 2,3\rbrace,X \rbrace$. One can verify that it is infact a topological quandle.
 \begin{lemma}
 \label{lem:coh for non ind}
 $\bar{H}_{0}^{Q}(X)=\mathbb{Z}^{3}, \bar{H}_{1}^{Q}(X)=\mathbb{Z}^{2}, H^{0}_{Q}(X)=\mathbb{Z}^{3}\, \text{and} \: H^{1}_{Q}(X)=\mathbb{Z}^{2}$
 \end{lemma}
 \begin{proof}
 The chain groups of $X$ are as follows.\\
\begin{equation*} 
\begin{split}
\bar{C}_{0}^{Q}(X) =& \mathbb{Z}_{\sigma_{1}} \oplus \mathbb{Z}_{\sigma_{2}} \oplus \mathbb{Z}_{\sigma_{3}}\\
 \bar{C}_{1}^{Q}(X) =& \mathbb{Z}_{\sigma_{\left[2,3 \right] }} \oplus \mathbb{Z}_{\sigma_{\left[3,2\right]}} \\
 \bar{C}_{2}^{Q}(X) =& \mathbb{Z}_{\sigma_{\left[2,3,2 \right] }} \oplus \mathbb{Z}_{\sigma_{\left[3,2,3\right]}}
 \end{split}
 \end{equation*}
 Also 
 \begin{equation} 
 \label{eqn:partial of 3quandle}
 \begin{split}
\partial_{1}(\sigma_{\left[ 2,3\right] })=\partial_{1}(\sigma_{\left[ 3,2\right] })=0\\
\partial_{2}(\sigma_{\left[ 2,3,2\right] })=\partial_{2}(\sigma_{\left[ 3,2,3\right] })=0
\end{split}
  \end{equation} 
  Therefore $\bar{H}_{0}^{Q}(X)= \mathbb{Z}_{\sigma_{1}} \oplus \mathbb{Z}_{\sigma_{2}} \oplus \mathbb{Z}_{\sigma_{3}}, \bar{H}_{1}^{Q}(X)=\mathbb{Z}_{\sigma_{\left[2,3 \right] }} \oplus \mathbb{Z}_{\sigma_{\left[3,2\right]}}$.\\
  If $f\in \bar{C}^{1}_{Q}(X)$, then $f=c_{(2,3)}\chi_{\sigma_{\left[2,3 \right] }}+c_{(3,2)}\chi_{\sigma_{\left[3,2\right]}}$. Here $\bar{Z}^{1}_{Q}(X)=\bar{C}^{1}_{Q}(X)$ as $\delta^{1}:\bar{C}_{1}^{Q}(X)\rightarrow\bar{C}_{2}^{Q}(X)$ is a zero map as $\partial_{2}$ is a zero map. Note that $\delta^{0}\chi_{\beta_{i}}=0$ for $i=1,2,3$ from equation \ref{eqn:partial of 3quandle}. Therefore $\bar{Z}^{0}_{Q}(X)=\bar{C}^{0}_{Q}(X)$ and $\bar{B}^{1}_{Q}(X)=0$. Hence the lemma follows.
  \end{proof}
  \begin{lemma}
  \label{lem:coh for ind}
  Let $Y$ be a topological quandle with the underlying set and quandle operation as above but with indiscrete topology. Then  $\bar{H}^{1}_{Q}(Y)=\mathbb{Z}^{2}$
  \end{lemma}
  \begin{proof}
  If $f =\displaystyle{\sum_{i,j \in Y, i \neq j}} c_{(i,j)}\chi_{(i,j)} \in Z^{2}_{Q}(Y)$ then
  \begin{equation*}
  c_{(p,r)}-c_{(p \triangleright q,r)}-c_{(p,q)}+c_{(p \triangleright r,q \triangleright r)}=0
  \end{equation*}
  for \begin{equation*}
  \begin{split}
  (p,q,r)\in \lbrace(1,2,3),(1,3,2),(2,1,3),(2,3,1),(3,1,2),(3,2,1),\\
  (1,2,1),(1,3,1),(2,1,2),(2,3,2),(3,1,3),(3,2,3)\rbrace
  \end{split}
  \end{equation*} and where $c_{(i,i)}=0$ for $i\in Y$. This gives the condition
  \begin{equation*}
  \begin{split}
 c_{(2,3)}=c_{(3,2)}=0 \\
  c_{(1,2)}-c_{(1,3)}=0
 \end{split}
  \end{equation*}
  Therefore any cocycle $f$ can be written as $f=c_{(1,2)}(\chi_{(1,2)}+\chi_{(1,3)})+c_{(2,1)}\chi_{(2,1)}+c_{(3,1)}\chi_{(3,1)}$.
  Note that 
  \begin{equation*}
  \begin{split}
  \partial_{1}^{Q}(1,2)=\partial_{1}^{Q}(1,3)=\partial_{1}^{Q}(2,3)=\partial_{1}^{Q}(3,2)=0\\
  \partial_{1}^{Q}(3,1)=(3)-(2)\\
  \partial_{1}^{Q}(2,1)=(2)-(3)
  \end{split}
  \end{equation*}
  So 
  \begin{equation*}
  \begin{split}
  \delta \chi_{(1)}=0\\
   \delta \chi_{(2)}=\chi_{(2,1)}-\chi_{(3,1)}\\
   \delta \chi_{(3)}=\chi_{( 3,1)}-\chi_{(2,1)}
  \end{split}
  \end{equation*}
  Therefore $H^{2}_{Q}(Y)=\mathbb{Z}_{(12)+(1,3)} \oplus \mathbb{Z}_{(2,1)}$.
  Now because of theorem \ref{thm:cohomology} $\bar{H}^{1}_{Q}(Y)=H^{2}_{Q}(Y)$.
\end{proof}
Consider the topological quandle $R_{4}=\lbrace a_{1},a_{2},b_{1},b_{2}|a_{i}\triangleright a_{j}=a_{i},b_{i}\triangleright b_{j}=b_{i},a_{i}\triangleright b_{j}=a_{i+1},b_{i}\triangleright a_{j}=b_{i+1}\rbrace$, where $2+1$ is considered as $1$ for subscripts, with topology $\tau=\lbrace\emptyset,\lbrace a_{1},a_{2}\rbrace\,\lbrace b_{1},b_{2}\rbrace\,X \rbrace $. The underlying quandle here is dihedral quandle of four elements.
\begin{lemma}
\label{lemma:R4}
The 1$^{st}$ cohomology group of $R_{4}$, $\bar{H}^{1}_{Q}(R_{4})=\mathbb{Z}^{4}.$
\end{lemma}
\begin{proof}
The cochain groups of $R_{4}$ are as follows.
\begin{equation*} 
\begin{split}
\bar{C}^{0}_{Q}(R_{4}) =& \chi_{\sigma_{a_{1}}} \oplus \chi_{\sigma_{a_{2}}} \oplus \chi_{\sigma_{b_{1}}}\oplus \chi_{\sigma_{b_{2}}}\\
\end{split}
\end{equation*}
The path components of $R_{4}$ are $\lbrace a_{1},a_{2} \rbrace$ and $\lbrace b_{1},b_{2} \rbrace$.Therefore
\begin{equation*}
    \begin{split}
 \bar{C}^{1}_{Q}(R_{4}) =& \mathbb{Z}_{\sigma_{\left[a_{1},a_{2} \right] }} \oplus \mathbb{Z}_{\sigma_{\left[a_{2},a_{1}\right]}}\oplus \mathbb{Z}_{\sigma_{\left[b_{1},b_{2}\right]}}\oplus \mathbb{Z}_{\sigma_{\left[b_{1},b_{2}\right]}} \\
 \bar{C}^{2}_{Q}(R_{4}) =& \mathbb{Z}_{\sigma_{\left[a_{1},a_{2},a_{1} \right] }} \oplus \mathbb{Z}_{\sigma_{\left[a_{2},a_{1},a_{2}\right]}}\oplus \mathbb{Z}_{\sigma_{\left[b_{1},b_{2},b_{1}\right]}}\oplus \mathbb{Z}_{\sigma_{\left[b_{2},b_{1},b_{2}\right]}} \\
 \end{split}
 \end{equation*}
 The map $\delta^{1}:\bar{C}^{1}_{Q}(R_{4}) \rightarrow \bar{C}^{2}_{Q}(R_{4}) $ is zero since $\partial_{2}:\bar{C}_{2}^{Q}(R_{4}) \rightarrow \bar{C}_{1}^{Q}(R_{4})$ is zero. In a similar way the map $\delta^{0}:\bar{C}^{0}_{Q}(R_{4}) \rightarrow \bar{C}^{1}_{Q}(R_{4}) $ is zero since $\partial_{1}:\bar{C}_{1}^{Q}(R_{4}) \rightarrow \bar{C}_{0}^{Q}(R_{4})$ is zero. Therefore $\bar{H}^{1}_{Q}(R_{4})=\bar{C}^{1}_{Q}(R_{4}) =\mathbb{Z}^{4}$
\end{proof}
\section{State Sum Invariant}
In this section we consider finite topological quandles and the coefficient group $A$ will be abelian and multiplicative.
Similiar to the method of computing state invariants of knot diagrams using quandle cocycles, topological quandle 1-cocycles are used to construct knot invariants.
\begin{lemma}
Let $\lbrace X,\tau \rbrace$ be a topological quandle and $\lbrace X_{\alpha} \rbrace_{\alpha \in \Gamma}$ be its path components. Then each $X_{\alpha}$ is a topological subquandle.
\end{lemma}
\begin{proof}
Give $X_{\alpha}$ the subspace topology of $X$. The set $\beta_{y}(X_{\alpha})$ is path connected as $\beta_{y}$ is a homeomorphism. Therefore if $x$ and $y$ are in $X_{\alpha}$ , $\beta_{y}(y)$ and $\beta_{y}(x)$ are in same component. But $\beta_{y}(y)=y \in X_{\alpha}$. Therefore $\beta_{y}(x)=x\triangleright y \in X_{\alpha}$. Hence $X_{\alpha}$ is closed under the quandle operation $\triangleright$. Other conditions can be easily verified.
\end{proof}
Consider the colorings of $D$ by $X_{\alpha}$, where coloring of a knot diagram is what defined in \cite{carter2003quandle}. We call the colorings of $D$ by all path components of $X$ to be the $topological\: quandle \:colorings$ of $X$. Let $\vert D_{\alpha} \vert$ denote the number of colorings of $D$ by $X_{\alpha}$.
\begin{theorem}
$\displaystyle{\sum_{\alpha \in \Gamma}} \vert D_{\alpha} \vert$ is a knot invariant.
\end{theorem}
We define the weight at a crossing and state sum invariant of knots in the case of topological quandles similar to the case of quandle cocycles.
\begin{definition}
Let $\phi \in \bar{Z}^{1}_{Q}(X,A) $ be a  1-cocycle. The weight $B(\theta,\mathfrak{C}_{\alpha})$ at a crossing $\theta$ of a diagram colored by $X_{\alpha}$ is defined as follows: Let $\mathfrak{C}_{\alpha}$ be the coloring by $X_{\alpha}$. Let the over arc $r$ be labelled by $\mathfrak{C_{\alpha}}(r)=y$ and $\mathfrak{C_{\alpha}}(r_{1})=x$, where the normal vector of $r$ points from under arc $r_{1}$ to $r_{2}$. Then $B(\theta,\mathfrak{C}_{\alpha}):=\phi(\sigma_{\left[ x,y \right] })^{\varepsilon(\theta)}$ where $\sigma_{\left[ x,y \right]}$ is the path from $x$ to $y$ and $\epsilon(\theta)=1\: \text{or}\: -1$, if crossing is positive or negative, respectively.
\end{definition}
\begin{figure}[h]
\def\svgwidth{200px}
\centering 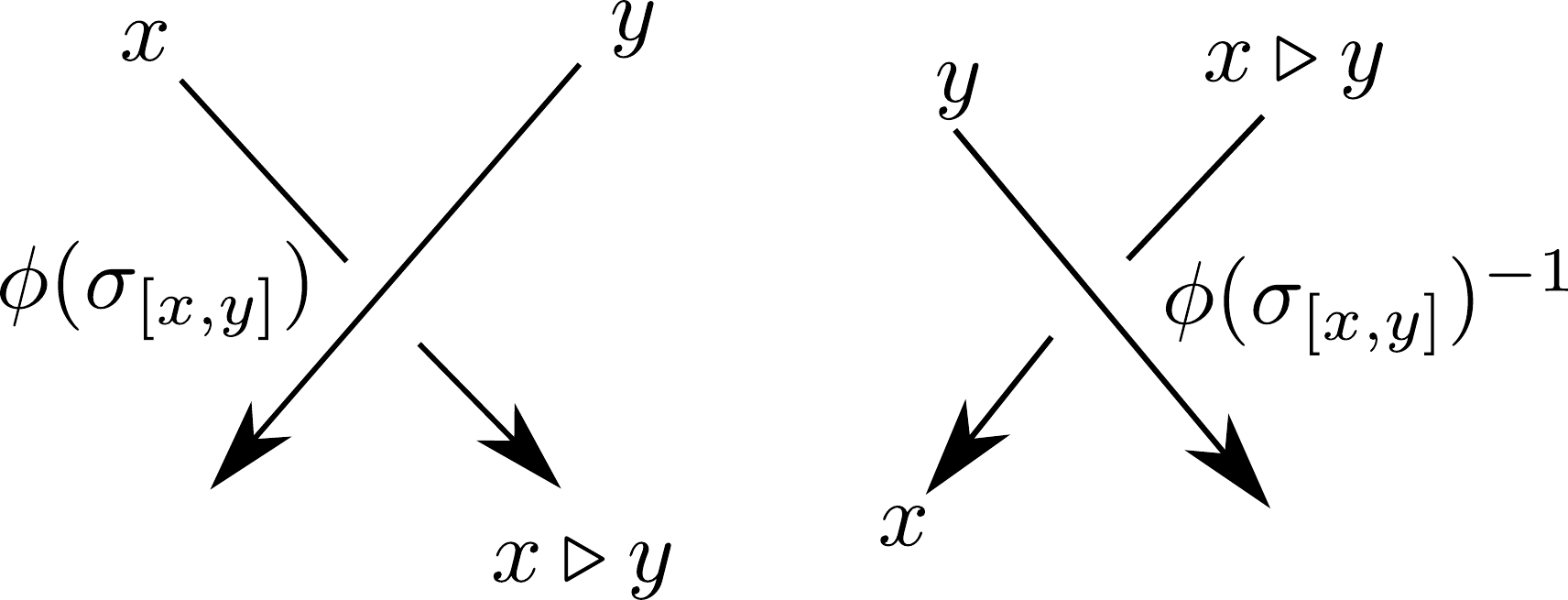
\centering \caption{weight of positive and negative crossings}
\label{figure:crossings}
 \end{figure}
\begin{definition}
Let $\phi \in \bar{Z}^{1}_{Q}(X,A) $. Then the topological state sum invariant  of a knot diagram is defined as $\displaystyle{\sum_{C}}\prod_{\theta} B(\theta,\mathfrak{C})$, which is sum of products of all crossings of a knot diagram colored by the path components $X_{\alpha}$s of $X$. We denote the topological state sum invariant of a knot diagram $K$ to be $\Phi(K)$
\end{definition}
\begin{lemma}
Let $p,q \;\text{and}\; r \in X_{\alpha}$. Then there exists $\beta_{\left[p,q,r \right] }:\triangle^{2}\rightarrow X_{\alpha}$ such that it's vertices map to $p,q \;\text{and}\; r$.
\end{lemma}
\begin{proof}
We know $X_{\alpha}$ is path connected. Suppose $\gamma_{1}$ is a path from $p$ to $q$ and $\gamma_{2}$ be a path from $q$ to $r$. Denote $\overline{\gamma_{1} \ast \gamma_{2}}$ to be the inverse of the path $\gamma_{1} \ast \gamma_{2}$. Consider the function $f:\partial(\triangle^{2})\rightarrow X_{\alpha}$ which is defined as in the following figure.
\begin{figure}[h]
\def\svgwidth{100px}
\centering 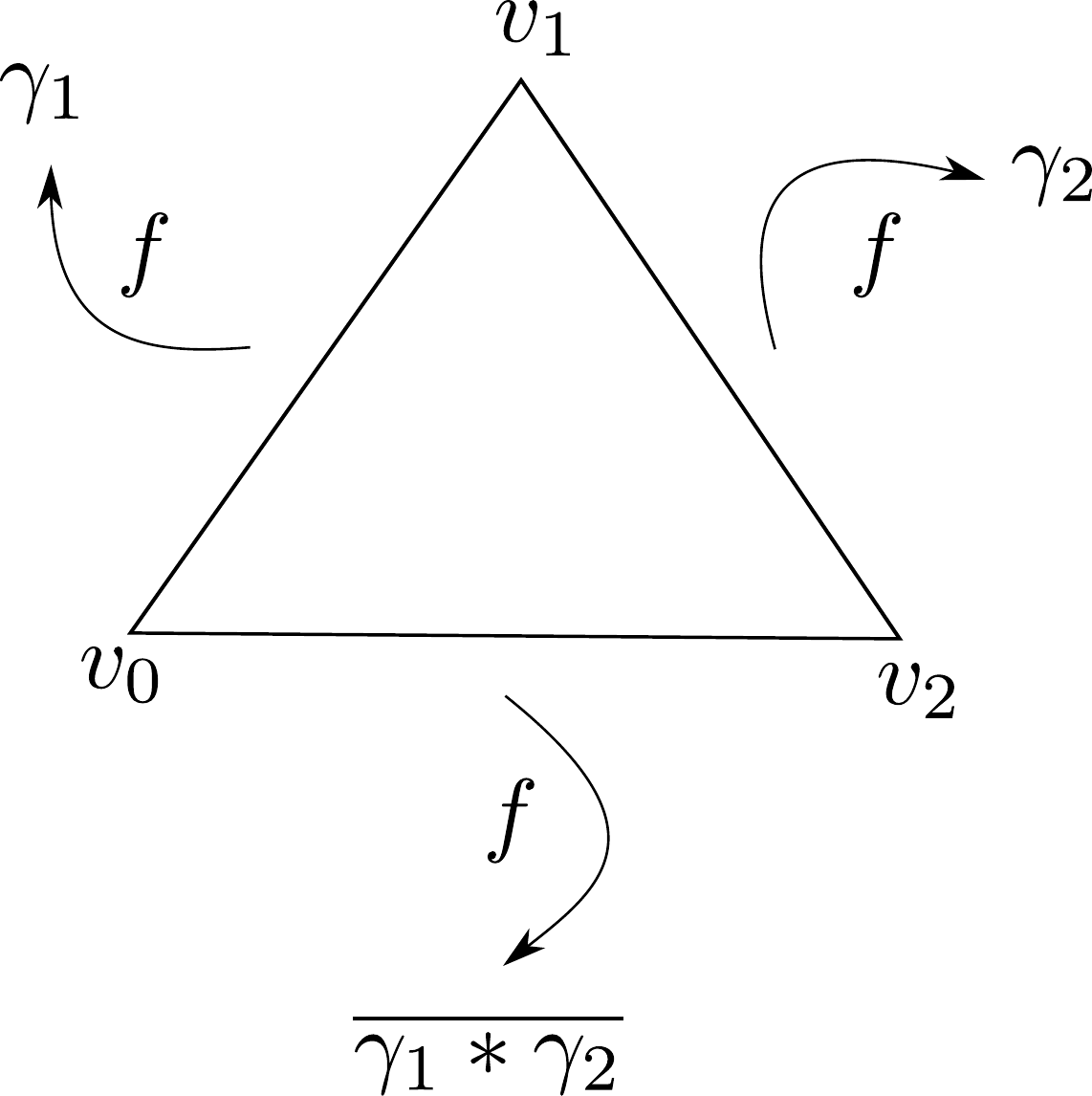
 \caption{ $\beta_{\left[p,q,r \right] }:\triangle^{2}\rightarrow X_{\alpha}$ mapping vertices to $p,q \; \text{and} \; r$}
 \end{figure}
 Here $f$ maps the edges $\left[v_{0},v_{1} \right] $ to $\gamma_{1}$, $\left[v_{1},v_{2} \right] $ to $\gamma_{2}$ and $\left[v_{2},v_{0} \right] $ to $\overline{\gamma_{1} \ast \gamma_{2}}$. Therefore $f$ is null homotopic and it has a continous extension from $\triangle^{2}$ to $X_{\alpha}$. This is our choice for $\beta_{\left[p,q,r \right] }:\triangle^{2}\rightarrow X_{\alpha}$.
\end{proof}
The above lemma gives the complete description of $\bar{C}_{2}^{R}(X)$. The group $\bar{C}_{2}^{R}(X)$ is a free abelian group generated by $2$ simplices whose vertices belong to same path component of $X$.
\begin{proposition}
Let $\phi \in \bar{Z}^{1}_{Q}(X,A)$. Then the partition function associated to $\phi$ is a knot invariant.
\end{proposition}
\begin{proof}
It's enough to show $\Phi(K)$ is invariant under Reidemiester moves.
\par The weight assigned to Reidemiester 1 move is $\phi(x,x)$ whose value is one and it won't change $\Phi(K)$.
\par There are two kinds of Reidemiester 2 moves, one with both strands having the same orientation and the second type with opposite orientations. In both cases the product of weights of a coloring remain the same before and after the move. 
\begin{figure}[h]
\def\svgwidth{150px}
\centering 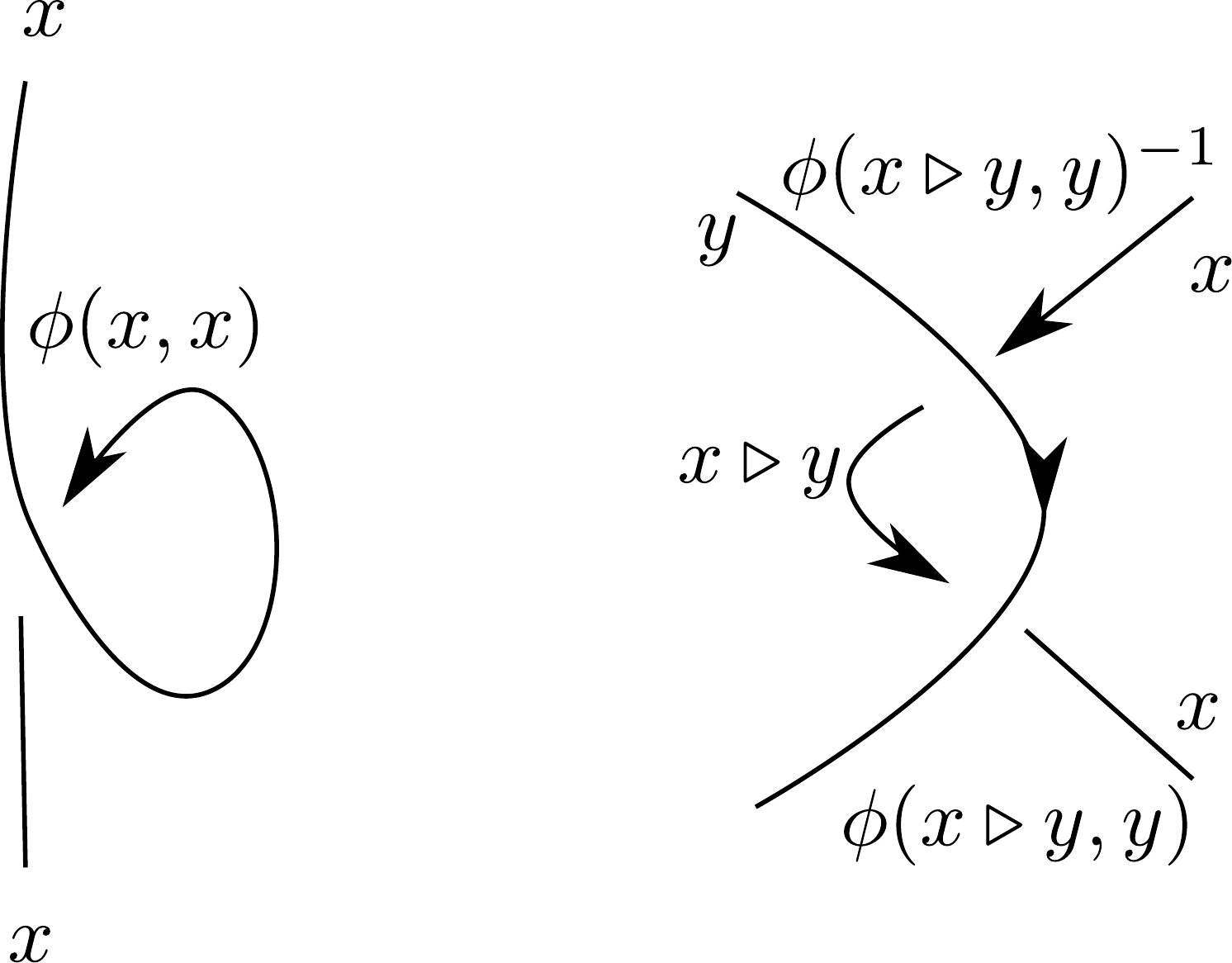
 \caption{weight of Reidemiester 1 and 2 move}
 \end{figure}
\par In \cite{polyak2010minimal}, by $theorem \:1.2$ we only have to show the Reidemiester 3 move depicted in figure \ref{figure:reidemeister3} won't change 
$\Phi(K)$, which is an easy verification.
\begin{figure}[h]
\def\svgwidth{200px}
\centering 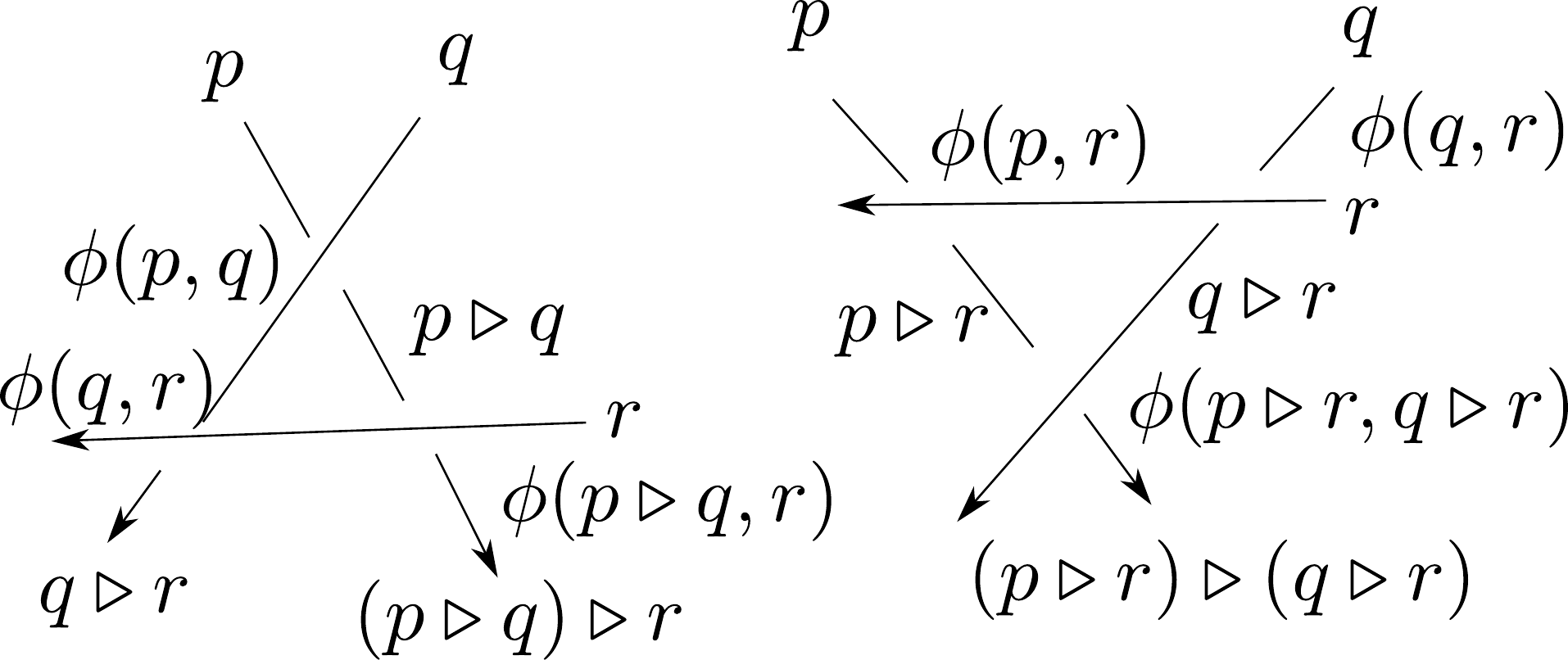
 \caption{weights before and after Reidemiester 3 move}
 \label{figure:reidemeister3}
 \end{figure}
\end{proof}
\begin{theorem}
Let $\phi,\: \phi' \in \bar{Z}^{1}_{Q}(X,A)$. If $\phi$ and $\phi'$ are cohomologous $($i.e, $\phi\cdot\phi'^{-1}=\delta \psi$ for some $0$ cochain $\psi\in \bar{C}^{0}_{Q}(X,A))$ then the corresponding state sum invariants $\Phi_{\phi}(K)$ and $\Phi_{\phi'}(K)$ are equal for any link $K$. In particular if $\phi$ is a coboundary then $\Phi_{\phi}(K)$ is equal to number of topological quandle colorings by $X$. 
\end{theorem}
\begin{proof}
If $\phi=\delta \psi$, then $\phi(\sigma_{[x,y]})=\psi(\sigma_{x})\psi(\sigma_{x\triangleright y})^{-1}$ and it follows that $\phi(\sigma_{[x,y]})^{-1}=\psi(\sigma_{x\triangleright y})\psi(\sigma_{x})^{-1}.$ For a link diagram colored by topological quandles, the weight at each crossing can be seen as weights assigned to the ends of under arcs at the crossing. See figure \ref{figure:crossings}. At a positive crossing  the weight $\psi(\sigma_{x})$ is assigned to the end of under arc labeled $x$ and the weight $\psi(\sigma_{x\triangleright y})^{-1}$ is assigned to the end of under arc labeled $x\triangleright y$. At a negative crossing  the weight $\psi(\sigma_{x})^{-1}$ is assigned to the end of under arc labeled $x$ and the weight $\psi(\sigma_{x\triangleright y})$ is assigned to the end of under arc labeled $x\triangleright y$. Each under arc of a colored diagram has only one colour, so the product of weights at ends of each arc is 1. Therefore for a colored diagram the product of weights of all crossings is 1. Hence both statements of the theorem hold.
\end{proof}
\begin{example}
Consider the two oriented hopf links with different orientations below.
\begin{figure}[h]
\def\svgwidth{200px}
\centering 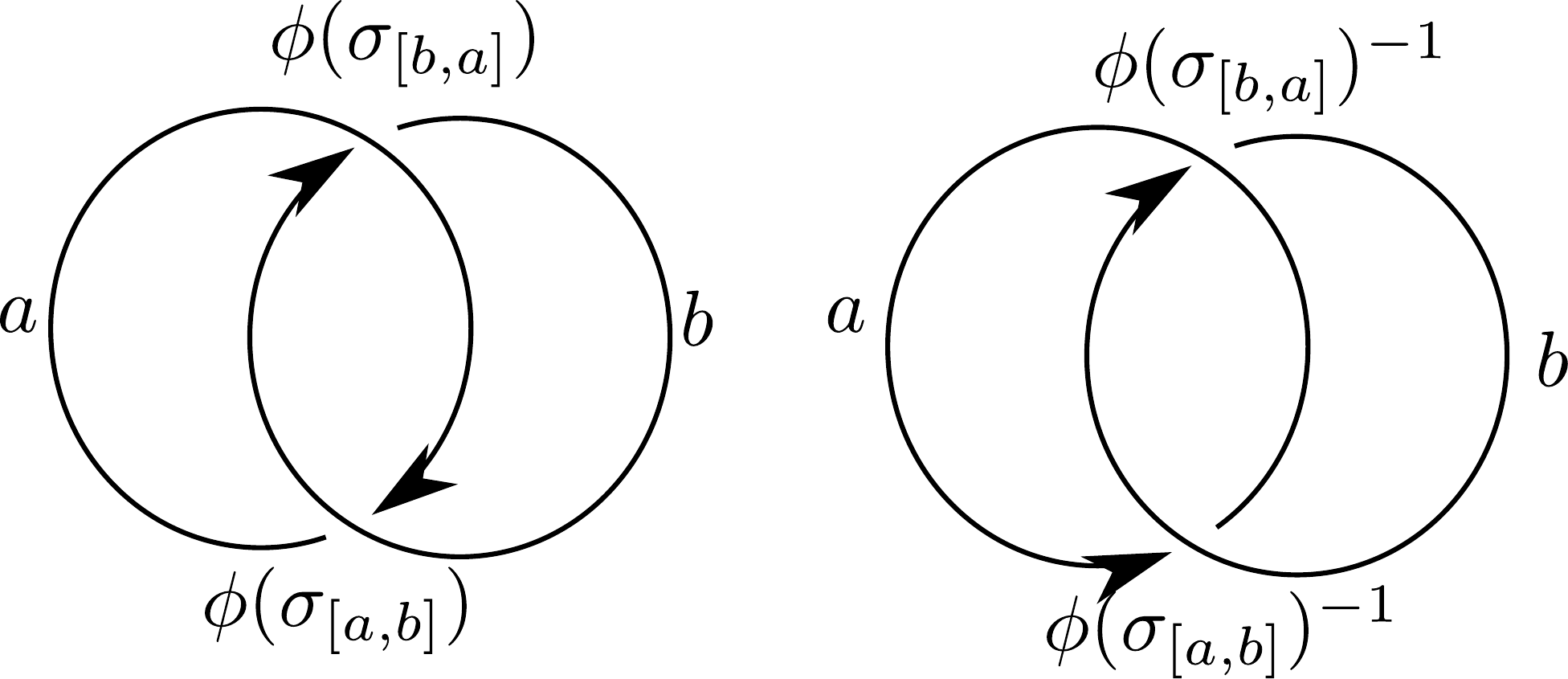
 \caption{Hopf links with different orientations}
 \label{figure:hopflink}
 \end{figure}
 Both the diagrams in figure \ref{figure:hopflink} have five colorings by the topological quandle given in lemma \ref{lem:coh for non ind} and lemma \ref{lem:coh for ind}, where $(a,b)\in \lbrace(1,1),(2,2),(3,3),(2,3),(3,2)\rbrace$ indicate color of the components. Suppose $\phi \in Z^{2}_{Q}(Y)$ as in lemma \ref{lem:coh for ind} then the topological state sum invariant is trivial. Therefore the state sum invariant calculated from any quandle 2-cocycle of underlying quandle structure is also trivial. But if we choose $\phi=\chi_{\sigma_{\left[2,3 \right] }} \in \bar{Z}^{1}_{Q}(X)$ as in lemma \ref{lem:coh for non ind} then the state sum invariant is $3+2t$ for the first hopf link and $3+2t^{-1}$ for the second one.  
\end{example}
\begin{example}
Consider the closure of the braid $\sigma_{1}^{6}\in B_{2}$ and the braid with opposite orientation for the second component as in figure \ref{figure:braid}.
\begin{figure}[h]
\def\svgwidth{300px}
\centering 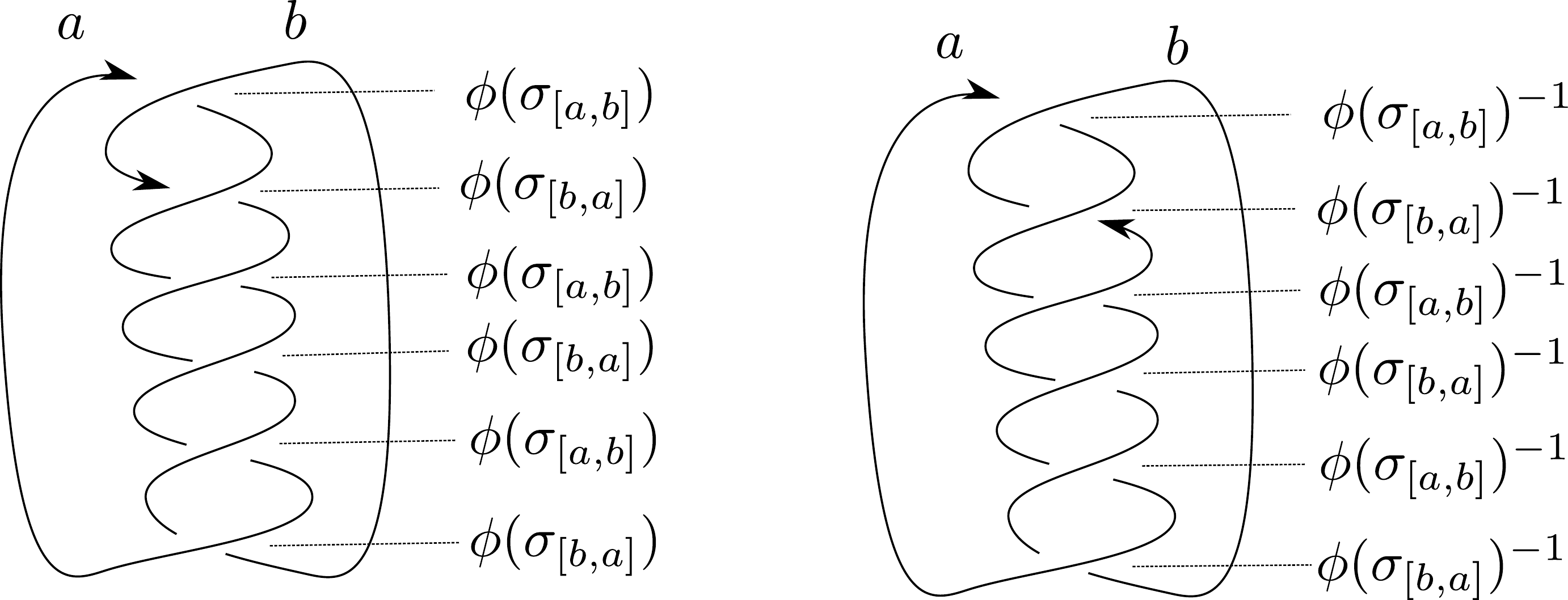
 \caption{closure of $\sigma_{1}^{6}$ with different orientations.}
 \label{figure:braid}
 \end{figure}
Both the figures below have same number of topological colorings by the topological quandle $R_{4}$ with the topology given in lemma \ref{lemma:R4}. Note that for a coloring of any of the given two diagrams every arc of each component should have the same color. It's easy to verify that the colorings 
of both diagrams are $(a,b)\in \lbrace(a_{1},a_{1}),(a_{2},a_{2}),(a_{1},a_{2}),(a_{2},a_{1}),(b_{1},b_{1}),\\(b_{2},b_{2}),(b_{1},b_{2}),(b_{2},b_{1})\rbrace$, where $a$ and $b$ denotes the color of $1^{st}$ and $2^{nd}$ components respectively. If we choose $\phi=\chi_{\sigma_{[a_{1},a_{2}]}}$ then $\Phi(K)$ is $2t^{3}+6$ for the first link and $2t^{-3}+6$ for the second link, whereas the state sum invariant obtained from any quandle 2-cocycle of the dihedral quandle of four elements is trivial\cite{carter2003quandle}.
\end{example}

\bibliography{ref}{}
\bibliographystyle{plain}
\medskip
\medskip

\end{document}